\documentclass{amsart}
\usepackage{amsmath,amssymb,graphicx}
\usepackage[latin1]{inputenc}
\usepackage[latin1]{inputenc}
\usepackage[english]{babel}
\usepackage{xcolor}

\newtoks\nazev
\newtheorem{thm}{Theorem}[section]
\newtheorem{theorem}[thm]{Theorem}

\newtheorem{claim}[thm]{Claim}
\newtheorem{prop}[thm]{Proposition}
\newtheorem{cor}[thm]{Corollary}
\newtheorem{example}[thm]{Example}
\newtheorem{question}[thm]{Question}

\theoremstyle{definition}
\newtheorem{definition}[thm]{Definition}

\def\eqn#1$$#2$${\begin{equation}\label#1#2\end{equation}}

\def\d{\,\mbox{\rm d}}

\def\th{{\leavevmode\setbox1=\hbox{t}%
  \hbox to \wd1{t\kern-0.6ex{\char039}\hss}}}%
\def\dh{{\leavevmode\setbox1=\hbox{d}%
  \hbox to 1.05\wd1{d\kern-0.4ex{\char039}\hss}}}%
\def\=#1{\if #1u{\accent23u}\else
\ifx #1d{\dh}\else \ifx #1t{\th}\else
 {\accent20 #1}\fi\fi\fi}
\def\'#1{\if #1i{\accent19\i}\else {\accent19 #1}\fi}

\def\N{\mathcal N}

\def\bd{\operatorname{Ja}}
\def\ep{\varepsilon}

\def\en{\mathbb N}
\def\er{\mathbb R}

\def\dist{\operatorname{dist}}
\def\dh{\,\widehat{\mbox{\rm d}}}

\def \conv {\operatorname{conv}}

\def \ext {\operatorname{ext}}

\def \reg {\partial _{\kern1pt\text{reg}}}

\def\wscl#1{\overline{#1}^{w^*}}

\newtoks\by
\newtoks\paper
\newtoks\book
\newtoks\jour
\newtoks\yr
\newtoks\pages
\newtoks\vol
\newtoks\publ
\newtoks\eds
\newtoks\proc
\newtoks\mathrev
\newtoks\web
\def\ota{{\hbox{???}}}
\def\cLear{\by=\ota\paper=\ota\book=\ota\jour=\ota\yr=\ota
\pages=\ota\vol=\ota\publ=\ota}
\def\endpaper{\the\by, \textit{\the\paper},
{\the\jour} \textbf{\the\vol} (\the\yr), \the\pages.\cLear}
\def\endbook{\the\by, \textit{\the\book}, \the\publ, \the\yr.\cLear}
\def\endprep{\the\by, \textit{\the\paper}, \the\jour.}
\def\endtoappear{\the\by, \textit{\the\paper}, to appear in \the\jour.}
\def\endprepkma{\the\by, \textit{\the\paper}, \the\jour, available on http://adela.\-karlin.\-mff.\-cuni.\-cz/\~{ }rokyta/\-preprint/\-index.php.\cLear}
\def\endproc{\the\by, \textit{\the\paper}, \the\book,
\the\publ, \the\yr, \the\pages.\cLear}
\def\endper{\the\by, \textit{personal communication}.\cLear}

\def\cC{{{\mathcal C}}}

\def\epsilon{\varepsilon}

\def\natnums{\mathbb N}
\def\reals{\mathbb R}

\def\N{\natnums}

\newcommand{\cke}{\operatorname{ck_E}}
\newcommand{\ck}{\operatorname{ck}}
\newcommand{\Jae}{\operatorname{Ja_E}}
\newcommand{\Ja}{\operatorname{Ja}}

\def\clust{\mathrm{clust}}
\newcommand{\co}{\operatorname{co}}

\newcommand{\neleq}{\rotatebox{45}{$\leq$}}
\newcommand{\seleq}{\rotatebox{325}{$\leq$}}

\begin{document}
\title{A quantitative version of James' compactness theorem}
\author{Bernardo Cascales, Ond\v{r}ej F.K. Kalenda and Ji\v{r}\'{\i} Spurn\'y}

\address{Depto de Matem\'aticas.
Universidad de Murcia. 30.100 Espinardo. Murcia, Spain}
\email{beca@um.es}

\address{Charles University\\Faculty of Mathematics and Physics\\
Department of Mathematical Analysis\\
Sokolovsk\'a~83\\
186~75 Praha~8\\
Czech Republic}

\email{kalenda@karlin.mff.cuni.cz}
\email{spurny@karlin.mff.cuni.cz}
\thanks{The research of B. Cascales  was supported by FEDER and MEC Project MTM2008-05396 and by Fundaci\'on S\'eneca  (CARM), project 08848/PI/08. The research of O. Kalenda and J. Spurn\'y is supported by the project MSM 0021620839 financed by MSMT and partly supported by the research grant GAAV IAA 100190901.}


\begin{abstract} We introduce two measures of weak non-compactness $\Jae$ and $\bd$ that quantify, via distances, the idea of boundary behind James' compactness theorem. These measures tell us, for a  bounded subset $C$ of a Banach space $E$  and for given $x^*\in E^*$, how far from $E$ or $C$ one needs to go to find  $x^{**}\in \overline{C}^{w^*}\subset E^{**}$ with $x^{**}(x^*)=\sup x^* (C)$. A quantitative version of James' compactness theorem is proved using  $\Jae$ and $\bd$, and in particular it yields the following result: {\it Let $C$ be a closed convex bounded subset of a Banach space $E$ and $r>0$. If there is an element $x_0^{**}$ in $\wscl C$ whose distance to $C$ is greater than $r$, then there is $x^*\in E^*$ such that each $x^{**}\in\wscl C$ at which $\sup x^*(C)$ is attained has distance to $E$ greater than $r/2$.}  We indeed establish that $\Jae$ and $\bd$ are equivalent to other measures of weak non-compactness studied in the literature.
We also collect particular cases and examples showing when the inequalities between the different measures of weak non-compactness can be equalities and when the inequalities are sharp.
\end{abstract}

\maketitle

\section{Introduction}

The celebrated James' compactness theorem says that a closed convex subset $C$ of a Banach space $E$ is weakly compact whenever each $x^*\in E^*$ attains its supremum on $C$, see~\cite{jam}. In particular, $E$ is reflexive whenever each $x^*\in E^*$ attains its norm at some point of the closed unit ball $B_E$ of $E$. In the present paper we prove a quantitative version of this theorem. Such a result not only fits into the recent research on quantitative versions of various famous theorems on compactness presented amongst others in \cite{ang-cas1,cas-alt11,FHMZ,sua-alt,sua,sua-san}, to which we relate our results here too, but also yields a strengthening of James' theorem itself. In particular we get the following result:

\begin{thm}\label{int-j-convex}
Let $E$ be a Banach space, $C\subset E$ a closed convex bounded
set which is not weakly compact. Let $0\le c<\frac12\dh(\wscl{C},C)$ be arbitrary. Then there is some $x^*\in E^*$ such that for any $x^{**}\in\wscl C$ satisfying $x^{**}(x^*)=\sup x^*(C)$ we have $\dist(x^{**},{E})>c$.
\end{thm}

This is our notation: if $A$ and $B$ are nonempty subsets of a Banach space $E$, then $d(A,B)$ denotes the
\emph{usual $\inf$ distance} between $A$ and $B$  and the \emph{Hausdorff non-symmetrized distance} from $A$ to $B$ is defined by
    $$
        \dh(A,B)= \sup\{d(a,B): a\in A\}.
    $$
Notice that $\dh(A,B)$ can be different from $\dh(B,A)$ and that $\max\{\dh(A,B),\dh(B,A)\}$ is the Hausdorff distance between $A$ and $B$.
Notice further that $\dh(A,B)=0$ if and only if $A\subset\overline B$ and that
    \begin{equation*}
        \label{dhinf}
        \dh(A,B)=\inf \{\ep>0: A\subset B+\ep B_E\}.
    \end{equation*}
Let us remark that we consider the space $E$ canonically embedded into its bidual $E^{**}$ and that
by $\wscl{C}$ we mean the weak* closure of $C$ in the bidual $E^{**}$.

\medskip
When applying Theorem~\ref{int-j-convex} for $c=0$ we obtain the classical James' compactness theorem. Our results in this paper go beyond Theorem~\ref{int-j-convex}. We should stress that  what we really do in this paper is to introduce several measures of weak non-compactness in Banach spaces related to distances to boundaries and then study their relationship with other well known measures of weak non-compactness previously studied. Our main result is Theorem~\ref{prop-pryce}. Combination with known or easy results gives Corollary~\ref{cor:AllPossibleInequalities}.
Theorem~\ref{int-j-convex} is then an immediate consequence.

\medskip
The quantities that we introduce are the following:

\begin{definition}  Given a bounded subset $H$ of a Banach space $E$ we define:
\begin{align*}
  \Jae(H)=\inf \{\epsilon > 0: \text{ for every } x^*\in E^*, \text{ there is } x^{**}\in
  \overline{H}^{w^*} \\ \text{ such that }x^{**}(x^*)=\sup x^*(H) \text{ and } d(x^{**},E)\leq \epsilon\}
\end{align*}
and
\begin{align*}
  \bd(H)=\inf \{\epsilon > 0: \text{ for every } x^*\in E^*, \text{ there is } x^{**}\in
  \overline{H}^{w^*} \\ \text{ such that }x^{**}(x^*)=\sup x^*(H) \text{ and } d(x^{**},H)\leq \epsilon\}.
\end{align*}
\end{definition}

Note that the definition of $\bd(H)$ is clearly inspired by the notion of a {\em boundary} that is hidden in James' theorem. Recall that if  $Y$ is a Banach space and $K\subset Y^*$ is a convex weak*-compact set, then a subset $B\subset K$ is called a {\it boundary} of $K$ if for each $y\in Y$ there is $b^*\in B$ such that
        \[
            b^*(y)=\sup_{k^*\in K} k^*(y).
        \]

James' compactness theorem can be rephrased now in the following way: {\em let $E$ be a Banach space and $C\subset E$ a bounded closed convex set; if $C$ is a boundary of $\wscl{C}$, then $C$ is weakly compact.}

We will study the relationship of $\Jae(C)$ and $\bd(C)$ to other quantities measuring weak non-compactness of $C$. The two most obvious quantities of this kind are	
$\dh(\wscl{C},C)$ and $\dh(\wscl{C},E)$. We stress that these two quantities can be different (see, e.g. examples in Section~\ref{examples}). The first one can be called `measure of weak non-compactness' of $C$, the other one can be called `measure of relative weak non-compactness' of $C$.

Using the notation introduced above, Theorem~\ref{int-j-convex} says that the inequality $\Jae(C)\ge\frac12\dh(\wscl C,C)$ holds for any closed convex bounded subset $C$ of a Banach space $E$.

In the following section we introduce several other quantities measuring weak non-compactness and sum up easy inequalities among them. In Section~\ref{S-pryce}
we formulate and prove our main result. As a corollary we obtain that all considered quantities measuring weak non-compactness are equivalent.

In Section~\ref{S-krein} we discuss the relationship to the quantitative version of Krein's theorem. Section~\ref{examples} contains examples showing that most of the inequalities are sharp. In the final section we study some particular cases in which some of the inequalities become equalities.

\section{Measures of weak non-compactness}\label{S-measures}

In this section we define and relate several quantities measuring weak non-compactness of a bounded set in a Banach space. Such quantities are called {\it measures of weak non-compactness}.
Measures of non-compactness or weak non-compactness
have been successfully applied to study of compactness, in operator theory, differential
equations and integral equations, see for instance
\cite{ang-cas1,ast-tyl,blasi,cas-alt11,FHMZ,sua-alt,sua,sua-san,kry,kr-pr-sc}. An axiomatic approach to measures of weak non-compactness may be found in \cite{ban-mar,kr-pr-sc}. But many of the natural quantities do not satisfy all the axioms, so we will not adopt this approach. Anyway, there is one property which should be pointed out: A measure of weak non-compactness should have
value zero if and only if the respective set is relatively weakly compact.

\medskip
Let $(x_n)$ be a bounded sequence in a Banach space $E$. We define
$\clust_{E^{**}}((x_n))$ to be the set of all cluster points of this sequence  in $(E^{**},w^*)$, i.e.
    $$
        \clust_{E^{**}}((x_n))=\bigcap_{n\in\N}\overline{\{x_m:m>n\}}^{w^*}.
    $$

Given a bounded subset $H$ of a Banach space $E$ we define:
 $$
            \gamma (H)= \sup\{|\lim_n\lim_m x^*_m(x_n)- \lim_m\lim_n x^*_m(x_n)|:
            (x^*_m)\subset B_{E^*}, (x_n) \subset H
            \},
 $$
assuming the involved limits exist,
 $$
           \cke (H)= \sup_{(x_n)\subset H}\d(\clust_{E^{**}}((x_n)),E), \ \ \  \ck (H)= \sup_{(x_n)\subset H}\d(\clust_{E^{**}}((x_n)),H).
 $$

Properties of  $\gamma$ can be found in~\cite{ang-cas1,ast-tyl,cas-alt11,FHMZ,kr-pr-sc} whereas $\cke$ can be found in~\cite{ang-cas1} -- note that $\cke$ is denoted as $\ck$ in that paper; do not mistake  it for $\ck$ above.

So, for a bounded set $H\subset E$ we have the following quantities measuring  weak non-compactness:
$$\dh(\wscl H,H),\dh(\wscl H,E),\ck(H),\cke(H),\gamma(H),\Ja(H),\Jae(H).$$
Let us stress on the different nature of these quantities:

First, the quantities
$\dh(\wscl H,H)$, $\ck(H)$, $\gamma(H)$ and $\Ja(H)$ do not depend directly on the space $E$. More exactly, if $F$ is a Banach space and $H\subset E\subset F$, where $E$ is a closed linear subspace of $F$ and $H$ a bounded subset of $E$, then these quantities are the same, no matter whether we consider $H$ as a subset of $E$ or as a subset of $F$. This is trivial for $\dh(\wscl H,H)$, and $\ck(H)$ and follows from the Hahn-Banach extension theorem for
$\gamma(H)$ and $\Ja(H)$.

On the other hand, the quantities $\dh(\wscl H,E)$, $\cke(H)$ and $\Jae(H)$ may decrease if the space $E$ is enlarged. More exactly, if $H\subset E\subset F$ are as above, then it may happen that $\dh(\wscl H,F)<\dh(\wscl H,E)$ and similarly for the other quantities (see examples in Section~\ref{examples}).

Since we are interested in James' compactness theorem, the most important case for us is the case of a closed convex bounded set $H$. Nonetheless, we define the quantities for an arbitrary bounded set and formulate results as general as possible. Anyway, such generalization do not yield really new results in view of the following proposition.

\begin{prop}\label{prop-nonconvex} Let $E$ be a Banach space and $H\subset E$ a bounded subset.
\begin{itemize}
	\item[(i)] All the above defined quantities have the same value for $H$ and for $\overline{H}$.
	\item[(ii)] The quantities $\dh(\wscl H,E)$, $\Jae(H)$ and $\gamma(H)$ have the same value for $H$ and for the weak closure of $H$.
	\item[(iii)] $\Jae(\co H)\le\Jae(H)$ and $\gamma(\co H)=\gamma(H)$.
\end{itemize}
\end{prop}

\begin{proof} The assertion (i) is obvious. Let us proceed with the assertion (iii).
The first inequality is trivial. The second equality is not easy at all, it
is proved in \cite[Theorem 13]{FHMZ} -- see \cite[Theorem 3.3]{cas-alt11} for a different proof.

Finally, let us show (ii). The case of $\gamma(H)$ follows from (i) and (iii). The other cases are trivial.
\end{proof}

As for the quantities not covered by this proposition it seems not to be clear whether $\cke(H)$ has the same value for $H$ and for the weak closure of $H$.
The quantities $\cke(H)$ and $\dh(\wscl H,E)$ may increase when passing to $\co H$: this follows from results of \cite{sua} and \cite{sua-alt}, see Example~\ref{exa-ghm}.

We do not know whether the quantity $\Jae(H)$ may really decrease when passing to $\co H$. This question seems not to be easy. Indeed, in view of the obvious
inequalities $\Jae(\co H)\le\Jae(H)$, $\cke(H)\le\cke(\co H)$ and taking into account  $\Jae(H)\le \cke(H)$ (see Proposition~\ref{int-ineq-known} below), if we had  $\Jae(\co H)<\Jae(H)$ then we would conclude that $\Jae(\co H)<\cke(\co H)$. The only example of a convex set $C$ satisfying $\Jae(C)<\cke(C)$ known to us is given in Example~\ref{exa-ghm} below and it seems that it cannot be easily improved.

As for the quantities $\dh(\wscl H,H)$, $\ck(H)$ and $\Ja(H)$ -- they are natural in case of a convex set $H$. If $H$ is not convex, they are not measures of weak non-compactness in the above sense since they may be strictly positive even if $H$ is relatively weakly compact. This is witnessed by Example~\ref{exa-nonconvex} below.

The following proposition sums up the easy inequalities.

\begin{prop}\label{int-ineq-known}
Let $E$ be a Banach space.
\begin{itemize}
	\item Let $H\subset E$ be a bounded set. Then the following inequalities hold true:
	$$\Jae(H)\le\cke(H)\le\dh(\wscl H,E)\le\gamma(H).$$
\item Let $C\subset E$ be a convex bounded set. Then the following inequalities hold true:
\begin{equation*}
\begin{array}{ccccccccc}
&&\ck_E(C)&\le&\dh(\wscl C,E)&&&& \\
&\neleq&&\seleq&&\seleq&&& \\
\Jae(C)&	\le&\Ja(C)&\le&\ck(C)&\le & \dh(\wscl C,C)&\le& \gamma(C). \end{array}
\end{equation*}
\end{itemize}
\end{prop}

\begin{proof}  Let us start by the first part. The inequality $\cke(H)\le\dh(\wscl H,E)$ is trivial. The inequality $\dh(\wscl{H},E)\le \gamma(H)$ is proved in \cite[Proposition 8(ii)]{FHMZ}, see also~\cite[Corollary 4.3]{cas-alt11}. Let us show that $\Jae(H)\le\cke(H)$.

Note first that if $\Jae(H)=0$ then inequality $0\leq \cke(H)$ trivially holds. Assume that $0< \Jae(H)$ and take an arbitrary $0<\epsilon<\Jae(H)$. By definition there is $x^*\in E^*$ such that for any $x^{**}\in \wscl H$ with $x^{**}(x^*)=\sup x^*(H)$ we have that $\epsilon<\d(x^{**},E)$. Fix a sequence $(x_n)$ in $H$ satisfying $\sup x^*(H)=\lim_n x^*(x_n)$.
Then each weak* cluster point $x^{**}$ of $(x_n)$ satisfies $x^{**}(x^*)=\sup x^*(H)$,  hence $\epsilon\leq  \d(\clust_{E^{**}}((x_n)),E)$ and therefore $\epsilon \leq \cke(H)$. This finishes the proof for $\Jae(H)\leq \cke(H)$.

Now let us proceed with the second part. All inequalities are obvious but $\Jae(C)\leq \ck_E(C)$, $\Ja(C)\leq \ck(C)$ and $\dh(\wscl{C},C)\le \gamma(C)$.
The first one follows from the first part. The second one can be proved in the same way.

Now we prove that $\dh(\wscl{C},C)\le \gamma(C)$. Suppose that $r>0$ is such that $\dh(\wscl C ,C)>r$. Fix $x^{**}\in \wscl C$ such that $\dist(x^{**},C)>r$. By the Hahn-Banach separation theorem there is $x^{***}\in X^{***}$ with $\|x^{***}\|=1$ and $s\in\er$ such that
    \begin{equation}\label{1}
        x^{***}(x^{**})>s+r>s>\sup_{x\in C} x^{***}(x).
    \end{equation}
We will construct by induction two sequences $(x_n)$ in $C$ and
$(x^*_n)$ in $B_{E^*}$ such that the following conditions are satisfied
for each $n\in\en$:
\begin{itemize}
    \item [(i)] $x^{**}(x^*_n)>s+r$,
    \item [(ii)] $x^*_n(x_m)  < s$ for $m<n$,
    \item [(iii)] $x^*_m(x_n) > s+r$ for $m\le n$.
\end{itemize}

By (\ref{1}) and the Goldstine theorem we can choose $x^*_1$ satisfying (i).
Now suppose that
$n\in\en$ is such that $x^*_m$ for $m\le n$ and $x_m$ for $m<n$
satisfy (i)--(iii).  Using that (i) holds for $x^*_1,\dots,x^*_n$ and that $x^{**}\in \wscl C$, we can choose
$x_n\in C$ satisfying (iii). Further, by (\ref{1}) and the Goldstine theorem
 we can find $x^*_{n+1}\in B_{E^*}$ satisfying (i) and (ii). This completes the
construction.

By passing to subsequences we may assume that  $\lim_n x^*_n(x_m)$ exists for all $m\in\en$ and that $\lim_m x^*_n(x_m)$ exists for all $n\in \en$ and (ii) and (iii) are satisfied. By taking further subsequences we may assume also that the limits $\lim_n\lim_mx^*_n(x_m)$ and $\lim_m\lim_nx^*_n(x_m)$ exist and that again and (ii) and (iii) are satisfied. By the construction we get
$$\lim_n\lim_mx^*_n(x_m)\ge s+r\mbox{\quad and\quad }\lim_m\lim_nx^*_n(x_m)\le s,$$
hence $\gamma(C)\ge r$. This completes the proof.
\end{proof}

We note that in the second part of the proposition above we only have to use the convexity of $C$ to prove the inequality $\dh(\wscl{C},C)\le \gamma(C)$; the rest of the inequalities  hold for an arbitrary bounded set. But for non-convex sets only the first part is interesting. This is witnessed by the following example which shows in particular the failure of the inequality $\dh(\wscl{C},C)\le \gamma(C)$ if $C$ is not convex.
	
\begin{example}\label{exa-nonconvex}
Let $E=c_0$ or $E=\ell_p$ for some $p\in(1,\infty)$. Let $H=\{e_n:n\in\en\}$,
where $e_n$ is the canonical $n$-th basic vector. Then $H$ is relatively weakly compact, hence $\dh(\wscl H,E)=\cke(H)=\Jae(H)=\gamma(H)=0$. However,
$\Ja(H)=\ck(H)=\dh(\wscl H,H)=1$.
\end{example}

\begin{proof} As the sequence $(e_n)$ weakly converges to $0$, $H$ is relatively weakly compact. This finishes the proof of the first part. Moreover,
$\wscl H$ is in fact the weak closure of $H$ in $E$ and is equal to $H\cup\{0\}$.
Thus clearly $\dh(\wscl H,H)=\ck(H)=1$. Finally, to show $\Ja(H)\ge 1$, consider $x^*\in E^*$ represented by the sequence $(-\frac1{2^n})_{n=1}^\infty$ in the respective sequence space. Then $\sup x^*(H)=0$ and the only point in $\wscl H$ at which the supremum is attained is $0$. The observation that $d(0,H)=1$ completes the proof.
\end{proof}

We remark that, for non-convex $H$, it is more natural to consider the quantity $\dh(\wscl H,\co H)$ instead of using $\dh(\wscl H,H)$  (cf. Section~\ref{S-krein}). Similar versions of other quantities can be studied as well.

\section{Quantitative versions of James' theorem}\label{S-pryce}

This section is devoted to the proof of the main results of this paper. In the course of the proof we use a proof of James' compactness theorem due to J.\,D.~Pryce in \cite{pryce}.

\begin{theorem}\label{prop-pryce} Let $E$ be a Banach space and $H\subset E$ a bounded subset. Then
$$\frac12\gamma(H)\leq \Jae(H).$$
\end{theorem}

\begin{proof}
Assume that $\gamma(H)>r$ for some $r>0$.
We denote by:
\begin{itemize}
\item $F$ the space of all {norm} continuous positive homogenous real-valued functions on $E$, {{i.e.} continuous functions $f:E\to \reals$ satisfying $f(\alpha x)=\alpha f(x)$, $\alpha \geq 0$ and $x\in E$.}
\item $p(f)= \sup f(H)$, $f\in F$,
\item $P(f)=\sup |f|(H)$, $f\in F$.
\end{itemize}
Then $p$ is a sublinear functional  and $P$ is a {seminorm} on $F$.

Let $(f_i)\subset B_{E^*}$ and $(z_j)\subset H$ be {sequences} such that
\[	
\lim_i\lim_j f_i(z_j)-\lim_j\lim_i f_i(z_j)>r
\]
and all the limits involved exist. By omitting finitely many elements of $(f_i)$ we may assume that
    \begin{equation}\label{eq:ToBeUsedLemma15.3}
        \lim_j f_i(z_j)-\lim_j\lim_i f_i(z_j)>r,\quad i\in \en.
    \end{equation}
Hence for every $i\in \en$ there exists $j_0\in\en$ such that
\[
f_i(z_j)-\lim_i f_i(z_j)>r, \quad j\geq j_0.
\]

Let $X$ stand for the linear span of $\{f_i:i\in\en\}$.
As $X$ is separable in the seminorm $P$ and  the functionals $f_i$ are equicontinuous for the norm on $E$, it follows from \cite[Lemma~2]{pryce} that  we can suppose without loss of generality that
    \begin{equation}\label{pecko}
        p(f-\liminf_i f_i)=p(f-\limsup_i f_i)\quad \text{ for all }f\in X.
    \end{equation}

We denote
\[
K_n=\conv\{f_i: i\geq n\},\quad n\in\en,
\]
and thus {we} obtain
\[
F\supset E^*\supset X\supset K_1\supset K_2\supset \cdots .
\]

By the proof of \cite[Lemma~3]{pryce} {and bearing in mind the inequality~(\ref{eq:ToBeUsedLemma15.3})},  we obtain
\begin{equation}
\label{kjedna}
p(f-\liminf_i f_i)>r,\quad f\in K_1.
\end{equation}

Next we quote \cite[Lemma~4]{pryce}.
\begin{claim}\label{cl2}
Let $Y$ be a linear space, $\rho,\beta,\beta'$ be strictly positive numbers, $p$ be a sublinear functional on $Y$, $A\subset Y$ be a convex set and $u\in Y$ satisfy
\[
\inf_{a\in A} p(u+\beta a)>\beta \rho +p(u).
\]
Then there exists $a_0\in A$ such that
\[
\inf_{a\in A} p(u+\beta a_0+\beta' a)>\beta' \rho + p(u+\beta a_0).
\]
\end{claim}

This claim will be used to prove the following one which is a mild strengthening of \cite[Lemma 5]{pryce}. Let us fix $r'\in(0,r)$ arbitrary.

\begin{claim}
\label{cl3}
Let $(\beta_n)$ be a sequence of strictly positive numbers. Then there exists a sequence $(g_n)$ in $F$ such that  $g_n\in K_n$ for $n\in\en$ and
\begin{equation}
\label{nerp}
p\left(\sum_{i=1}^n \beta_i(g_i-\liminf_j f_j)\right)>\beta_n r'+p\left(\sum_{i=1}^{n-1} \beta_i(g_i-\liminf_j f_j)\right), \quad n\in\en.
\end{equation}
\end{claim}

\begin{proof}
The construction proceeds by induction. Let $f_0=\liminf_j f_j$.

 If $n=1$, we use Claim~\ref{cl2} for $u=0$, $\beta=\beta_1$, $\beta'=\beta_2$, $\rho=r'$, and $A=K_1-f_0$.
By \eqref{kjedna},
\[
\aligned
\inf_{g\in A} p(u+\beta g)&=\inf_{g\in A} \beta p(g)=\beta_1 \inf_{f\in K_1} p(f-\liminf_j f_j)\\
&>\beta_1 r'=\beta_1 r'+p(u),
\endaligned
\]
and hence Claim~\ref{cl2} gives the existence of $g_1\in K_1$  satisfying
\[
\inf_{f\in K_1} p(\beta_1(g_1-f_0)+\beta_2 (f-f_0))>\beta_2 r'+
p(\beta_1(g_1-f_0)).
\]
This finishes the first step of the construction.

Assume now that we have found $g_i\in K_i$, $i=1,\dots, n-1$, for some $n\in\en$, $n\ge2$, such that
\[
\inf_{f\in K_{n-1}-f_0} p\left(\sum_{i=1}^{n-1} \beta_i(g_i-f_0)+\beta_n f\right)>\beta_n r'+p\left(\sum_{i=1}^{n-1} \beta_i(g_i-f_0)\right).
\]
We use Claim~\ref{cl2} with $u=\sum_{i=1}^{n-1}\beta_i(g_i-f_0)$, $\beta=\beta_n$, $\beta'=\beta_{n+1}$, $\rho=r'$, and $A=K_n-f_0$. Since $K_n\subset K_{n-1}$, inductive hypothesis gives
\[
\inf_{f\in A} p(u+\beta f)\geq \inf_{f\in K_{n-1}-f_0} p(u+\beta f)>\beta_n r'+p(u).
\]
By Claim~\ref{cl2}, there exists $g_n\in K_n$ such that
\[
\inf_{f\in A} p\left(\sum_{i=1}^n \beta_i(g_i-f_0)+\beta_{n+1} f\right)>\beta_{n+1} r'+p\left(\sum_{i=1}^{n-1} \beta_i(g_i-f_0)+\beta_n (g_n-f_0)\right).
\]
This completes the inductive construction.

We have obtained elements $g_n\in K_n$, $n\in\en$, such that
\[
\inf_{g\in K_n} p\left(\sum_{i=1}^n \beta_i(g_i-f_0)+\beta_{n+1}(g-f_0)\right)>\beta_{n+1}r'+p\left(\sum_{i=1}^{n} \beta_i(g_i-f_0)\right).
\]
Since $g_{n+1}\in K_{n+1}\subset K_n$, this yields
\[
p\left(\sum_{i=1}^n \beta_i(g_i-f_0)+\beta_{n+1}(g_{n+1}-f_0)\right)>\beta_{n+1}r'+p\left(\sum_{i=1}^{n} \beta_i(g_i-f_0)\right).
\]
This finishes the proof.
\end{proof}

Let $\beta_i>0$, $i\in\en$, be chosen in such a way that $\lim_{n} \frac{1}{\beta_{n}}\sum_{i=n+1}^\infty \beta_i=0$.
Let $(g_n)$ be a sequence provided by Claim~\ref{cl3}. Since {for every $n\in\en$ we have that $g_n\in K_n\subset B_{E^*}$}, we can select a weak$^*$-cluster point $g_0\in B_{E^*}$  of $(g_n)$. By \cite[Lemma~6]{pryce}, we have the following observation.

\begin{claim}
\label{cl4}
For any $f\in X$, $p(f-g_0)=p(f-\liminf_n f_n)$.
\end{claim}

By Claim~\ref{cl4}, we can replace $\liminf_j f_j$ by $g_0$ in \eqref{nerp} and get the following inequalities
\begin{equation}
\label{nerg}
p\left(\sum_{i=1}^n \beta_i(g_i-g_0)\right)>\beta_n r'+p\left(\sum_{i=1}^{n-1} \beta_i(g_i-g_0)\right), \quad n\in\en.
\end{equation}

 We set $M=\sup\{\|x\|: x\in H\}$ and remark that $\|g_i-g_0\|\leq 2$, $i\in\en$.

We set $g=\sum_{i=1}^\infty \beta_i(g_i-g_0)$. Let $u\in \wscl{H}$ be an arbitrary point satisfying $g(u)=\sup g(H)$.
Then, for any $n\in\en$, we get from \eqref{nerg}
\[
\aligned
\sum_{i=1}^n\beta_i(g_i-g_0)(u)&=g(u)-\sum_{i=n+1}^\infty \beta_i(g_i-g_0){(u)}
\geq p(g)- 2M\sum_{i=n+1}^\infty\beta_i \\
&\geq p\left(\sum_{i=1}^n \beta_i(g_i-g_0)\right)-p\left(\sum_{i=1}^n \beta_i(g_i-g_0)-g\right)-2M\sum_{i=n+1}^\infty \beta_i\\
&\geq p\left(\sum_{i=1}^n \beta_i(g_i-g_0)\right)-4M \sum_{i=n+1}^\infty \beta_i\\
&>\beta_n r'+p\left(\sum_{i=1}^{n-1} \beta_i(g_i-g_0)\right)-4M \sum_{i=n+1}^\infty \beta_i\\
&\geq \beta_n r'+\sum_{i=1}^{n-1} \beta_i(g_i-g_0)(u)-4M \sum_{i=n+1}^\infty \beta_i.
\endaligned
\]
Hence
\[
(g_n-g_0)(u)\geq r'-4M \frac{1}{\beta_{n}}\sum_{i=n+1}^\infty \beta_i,\quad n\in\en,
\]
which gives
\begin{equation}
\label{limi}
\liminf_n (g_n-g_0)(u)\geq r'.
\end{equation}

{Let $v\in E$ be arbitrary.} Then $g_0(v)\geq \liminf_n g_n(v)$, which along with \eqref{limi} gives
\[
\aligned
r'&\leq \liminf_n g_n(u)-\liminf_n g_n(v)+g_0(v-u)\\
&\leq -\liminf_n (g_n(v)-g_n(u))+g_0(v-u)\\
&\leq 2\|v-u\|.
\endaligned
\]
By the definition of $\Jae(H)$ it follows $\Jae(H)\geq \frac12 r'$. Since $r$ satisfying $\gamma(H)>r$ and $r'\in(0,r)$ are arbitrary we conclude that  $\Jae(H)\geq \frac12 \gamma(H)$.
\end{proof}

As a consequence of Theorem~\ref{prop-pryce} we obtain that all measures of non-compactness that we have considered in this paper are equivalent. In other words, {\em all classical approaches used to study weak compactness in Banach spaces (Tychonoff's theorem, Eberlein's theorem, Grothendieck's theorem and James' theorem) are qualitatively and quantitatively equivalent.}

\begin{cor}\label{cor:AllPossibleInequalities} Let $E$ be a Banach space.
\begin{itemize}
	\item Let $H\subset E$ be a bounded set. Then the following inequalities hold true:
	$$\frac12\gamma(H)\le\Jae(H)\le\cke(H)\le\dh(\wscl H,E)\le\gamma(H).$$

 \item Let $C\subset E$ be a bounded convex set. Then the following inequalities hold true:
\begin{equation*}
\begin{array}{ccccccccccc}
&&&&\ck_E(C)&\le&\dh(\wscl C,E)&&&& \\
&&&\neleq&&\seleq&&\seleq&&& \\
\frac12\gamma(C)&\le&\Jae(C)&	\le&\Ja(C)&\le&\ck(C)&\le & \dh(\wscl C,C)&\le& \gamma(C). \end{array}
\end{equation*}
\end{itemize}
\end{cor}
\begin{proof} This result follows from Proposition~\ref{int-ineq-known} and Theorem \ref{prop-pryce}.
\end{proof}

The fact that the measures of weak non-compactness  $H\mapsto \dh(\wscl H,E)$, $\gamma$ and $\cke$ are equivalent can be found in~\cite{FHMZ} and~\cite{ang-cas1} with very different approaches.

In Section~\ref{examples} we offer several examples  showing that in the corollary above
any of the inequalities may become equality and that most of them may become strict.

\begin{cor} Let $E$ be a Banach space and $C\subset E$ be a closed convex bounded subset. Then $C$ is weakly compact provided $\Jae(C)=0$  (\emph{i.e.}, if for every $\epsilon >0$ and every $x^*\in X^*$ there is $x^{**}\in \overline{C}^{w^*}$ such that $x^{**}(x^*)=\sup x^*(C)$ and $d(x^{**},E)\leq \epsilon$).
\end{cor}

\section{Relationship to the quantitative version of Krein's theorem}\label{S-krein}

Let $E$ be a Banach space and $C\subset E$ be a bounded convex set.
Then $\ext \wscl C$, the set of extreme points of $\wscl C$, is a boundary for $\wscl C$. Therefore the following inequalities are obvious:

\begin{equation}\label{eq-extreme}
\begin{gathered}
\dh(\wscl{\ext \wscl C},C)\ge \dh(\ext\wscl C,C)\ge\Ja(C),\\
\dh(\wscl{\ext \wscl C},E)\ge \dh(\ext\wscl C,E)\ge\Jae(C).
\end{gathered}
\end{equation}
These inequalities enable us to prove the following statement.

\begin{cor}\label{cor-Krein}
Let $E$ be a Banach space and $H\subset E$ be a  bounded set. Then the following inequalities hold:
\begin{itemize}
	\item[(i)] $\dh (\wscl{\co H},E)\le 2 \dh(\wscl H,E)$,
	\item[(ii)] $\dh (\wscl{\co H},{\co H})\le 2 \dh(\wscl H,{\co H})$.
\end{itemize}
\end{cor}

\begin{proof} Set $C=\overline{\co H}$. Then $\ext \wscl C\subset \wscl H$,
so the inequalities follow from (\ref{eq-extreme}) and Corollary~\ref{cor:AllPossibleInequalities}.
\end{proof}

We remark that the assertion (i) was proved in \cite{FHMZ} and independently in \cite{sua} and \cite{cas-alt11}. In \cite{sua} and \cite{sua-alt} some examples are given which show that the inequality is optimal, i.e. the equality can take place if the quantities are non-zero.
However, these examples do not work for the assertion (ii). Hence, the following problem seems to be natural.

\begin{question} Let $E$ be a Banach space and $H\subset E$ a bounded set. Is it true that
$$\dh (\wscl{\co H},{\co H}) = \dh(\wscl H,{\co H}) \quad ?$$
\end{question}

The assertion (i) of Corollary~\ref{cor-Krein} is called in \cite{FHMZ} a quantitative version of Krein's theorem. Krein's theorem asserts that a closed convex hull of a weakly compact set is again weakly compact. This is the case when the quantities are $0$. In view of this also the assertion (ii) may be called a quantitative version of Krein's theorem. An interesting phenomenon is that there are examples showing that the inequality (i) is sharp but we do not know whether the inequality (ii) is sharp. Both examples showing sharpness of (i) are of similar nature: A set $H$ is constructed in a space $E_0$ such that
$\dh (\wscl{\co H},\co H) = \dh(\wscl H,{\co H})=1$.
Then the space $E_0$ is enlarged in a clever way to $E$ such that $\dh(\wscl{\co H},E)$ equals $1$ but $\dh(\wscl H,E)$ decreases to $\frac12$. If the space $E$ is enlarged even more, also the quantity $\dh(\wscl{\co H},E)$ will decrease to $\frac12$ and it will be no more a counterexample. Hence, a possible counterexample showing sharpness of (ii) should be of a quite different nature.

Moreover, one can show (although it is not obvious) that the answer to the above question is positive if $H$ is norm-separable. This is another indication of a great difference between (i) and (ii) as the example from \cite{sua-alt} is norm-separable (see Example~\ref{exa-ghm} below).

\section{Examples}\label{examples}

In this section we collect examples showing the sharpness of some of the inequalities that are collected in Corollary~\ref{cor:AllPossibleInequalities}.
We remark that unless all the quantities are zero, at least one of the inequalities must be strict. We stress again that the examples in this section show in particular that
any of the inequalities may become equality and that most of them may become strict.

\begin{example}\label{exa-c0} Let $E=c_0$ and $C=B_E$. Then $\gamma(C)=1$ and $\Jae(C)=1$.
Hence all other quantities are also equal to $1$.
\end{example}

\begin{proof} The equality $\gamma(C)=1$ follows from \cite[Example~2.7 and Theorem~2.8]{kr-pr-sc}. To show that $\Jae(C)\ge1$ take $x^*\in E^*$ represented by the sequence $(\frac1{2^n})_{n=1}^\infty$ in $\ell_1$. The only element of $\wscl C=B_{\ell_\infty}$ at which $x^*$ attains its supremum on $C$ is the constant sequence $(1)_{n=1}^\infty$ whose distance from $E$ is clearly $1$.
The rest now follows from Corollary~\ref{cor:AllPossibleInequalities}.
\end{proof}

\begin{example}\label{exa-ell1} Let $E=\ell_1$ and $C=B_E$. Then $\gamma(C)=2$ and $\dh(\wscl C,C)=1$. Hence all other quantities are equal to $1$.
\end{example}

\begin{proof} It is clear that $\dh(\wscl C,C)\le 1$.
Further, the inequality $\gamma(C)\ge 2$ is witnessed by sequences $(x_n)$ and $(x^*_n)$, where $x_n$ is the $n$-th canonical basic vector of $\ell_1$ and $x^*_n\in B_{\ell_\infty}$ is defined by
$$x^*_n(m)=\begin{cases} 1 & m\le n,\\ -1 & m>n.\end{cases}$$
The rest follows from Corollary~\ref{cor:AllPossibleInequalities}.
\end{proof}

\begin{example}\label{exa-c0omega} Let $E=C([0,\omega])$ and $C=\{x\in E: 0\le x\le 1\ \&\ x(\omega)=0\}$. Then $\dh(\wscl C,E)=\frac12$ and $\Ja(C)=1$.
Hence $\Ja_E(C)=\ck_E(C)=\frac12$ and $\ck(C)=\dh(\wscl C,C)=\gamma(C)=1$.
\end{example}

\begin{proof} Note that $E^*$ is canonically identified with $\ell_1([0,\omega])$ and $E^{**}$ with $\ell_\infty([0,\omega])$.

To show that $\dh(\wscl C,E)\le\frac12$ we observe that the constant function $\frac12$ belongs to $E$ and that $C\subset \frac12+\frac12 B_E$. Thus $\wscl C\subset \frac12+\frac12 B_{E^{**}}$.

Further, consider the element $x^*\in E^*=\ell_1([0,\omega])$ given by
$x^*(n)=\frac1{2^n}$ for $n<\omega$ and $x^*(\omega)=0$. Then the only element of $\wscl C$ at which $x^*$ attains its supremum on $C$ is $\chi_{[0,\omega)}$.
Its distance to $C$ of this element is clearly equal to $1$. Thus $\Ja(C)\ge 1$.

The rest follows from Corollary~\ref{cor:AllPossibleInequalities}.
\end{proof}

\begin{example}\label{exa-c00omega1} Let $E=C_0([0,\omega_1))$ and $C=\{x\in E: 0\le x\le 1\}$. Then $\dh(\wscl C,E)=1$ and $\ck(C)=\frac12$.
Hence $\Ja_E(C)=\Ja(C)=\ck_E(C)=\frac12$ and $\dh(\wscl C,C)=\gamma(C)=1$.
\end{example}

\begin{proof}
First note that the dual $E^*$ can be identified with $\ell_1([0,\omega_1))$ and the second dual $E^{**}$ with $\ell_\infty([0,\omega_1))$.

To show that $\dh(\wscl C,E)\ge1$ we note that the constant function $1$ belongs to $\wscl C$ and its distance to $E$ is $1$.

Next we will show that $\ck(C)\le\frac12$. Let $(x_n)$ be any sequence in $C$.
There is some $\alpha<\omega_1$ such that $x_n|_{(\alpha,\omega_1)}=0$ for each $n\in\en$. As the interval $[0,\alpha]$ is countable, there is a subsequence $(x_{n_k})$ which converges pointwise on $[0,\omega_1)$. The limit is an element of $\ell_\infty([0,\omega_1))=E^{**}$. Denote the limit by $x^{**}$. Then the sequence $(x_{n_k})$ weak* converges to $x^{**}$. Thus in particular $x^{**}\in \clust_{E^{**}}((x_n))$. Set $x=\frac12\chi_{[0,\alpha]}$. Then $x\in C$ and $\|x^{**}-x\|\le\frac12$ (as $0\le x^{**}\le 1$ and $x^{**}|_{(\alpha,\omega_1)}=0$). The inequality $\ck(C)\le\frac12$ now follows.

The rest follows from Corollary~\ref{cor:AllPossibleInequalities}.
\end{proof}

\begin{example} Let $E=C([0,\omega_1])$ and $C=\{x\in E: 0\le x\le 1\ \&\ x(\omega_1)=0\}$. Then $\dh(\wscl C,E)=\ck(C)=\frac12$ and $\dh(\wscl C,C)=1$.
Hence $\Ja_E(C)=\Ja(C)=\ck_E(C)=\frac12$ and $\gamma(C)=1$.
\end{example}

\begin{proof} We start similarly as in Example~\ref{exa-c0omega}:
Note that $E^*$ is canonically identified with $\ell_1([0,\omega_1])$ and $E^{**}$ with $\ell_\infty([0,\omega_1])$.

To show that $\dh(\wscl C,E)\le\frac12$ notice that the constant function $\frac12$ belongs to $E$ and that $C\subset \frac12+\frac12 B_E$. Thus $\wscl C\subset \frac12+\frac12 B_{E^{**}}$.

The inequality $\ck(C)\le \frac12$ can be proved in the same way as in Example~\ref{exa-c00omega1}. In fact, it follows from that example, since
$C_0([0,\omega_1))$ is isometric to $\{x\in E: x(\omega_1)=0\}$, and hence our set $C$ coincides with the set $C$ from Example~\ref{exa-c00omega1}.

Finally, $\dh(\wscl C,C)\ge 1$ as $\chi_{[0,\omega_1)}\in\wscl C$ and its distance from $C$ is equal to $1$.

The rest follows from Corollary~\ref{cor:AllPossibleInequalities}.
\end{proof}

\begin{example}\label{exa-ghm}
There is a Banach space $E$ and a closed convex bounded subset $C\subset E$ such that $\Jae(C)=\frac12$ and $\ck_E(C)=\Ja(C)=1$. Hence
$\ck(C)=\dh(\wscl C,E)=\dh(\wscl C,C)=\gamma(C)=1$.
\end{example}

\begin{proof} We use the example from \cite{sua-alt}. It is constructed there a set $K_0\subset[0,1]^{\en}$ and a free ultrafilter $u$ over $\en$ such that
(in particular) the following assertions are satisfied:

\begin{itemize}
	\item[(a)] $K_0$ consists of finitely supported vectors and is closed in the topology of uniform convergence on $\en$ but not in the pointwise convergence topology.
		\item[(b)] For each $x\in \overline{K_0}$ (the closure taken in the pointwise convergence topology) we have $\lim_u x(n) = 0$.
	\item[(c)] For each $x\in \overline{K_0}\setminus K_0$ there are infinitely many $n\in\en$ such that $x(n)=1$.
\end{itemize}

Let $E=\{x\in C(\beta\en) : x(u)=0\}$. We remark that $\beta\en$ is canonically identified with the space of ultrafilters over $\en$ and hence we have $u\in\beta\en$. Let us consider embedding $\kappa:\overline{K_0}\to E$ defined
by
$$\kappa(x)(p)=\lim_p x(n), \qquad p\in\beta\en,\, x\in\overline{K_0}.$$
By (b) it is a well defined mapping with values in $E$. Let $B=\kappa(K_0)$. Then $B$ is a bounded norm-closed subset of $E$. Set $C=\overline{\conv B}$.

It is proved in \cite{sua-alt} that $\dh(\wscl B,E)\le\frac12$. As $\wscl B$ contains extreme points of $\wscl C$, by (\ref{eq-extreme}) we get $\Jae(C)\le\frac12$.

In \cite{sua-alt} it is proved that $\dh(\wscl C,E)\ge1$.
We will show that even $\ck_E(C)\ge 1$. To do this it is enough to observe that $C\subset \{x\in E: x|_{\beta\en\setminus\en}=0\}$ (this follows from (a)). The latter space is isometric to $c_0$. As $c_0^*$ is separable, each element of $\wscl C$ is a weak* limit of a sequence from $C$. It follows that $\ck_E(C)=\dh(\wscl C,E)\ge1$.

By Corollary~\ref{cor:AllPossibleInequalities} it remains to prove that $\Ja(C)\ge 1$.
 To do that let us first recall that the dual to $E$ can be canonically identified with the space of all signed Radon measures on $\beta\en\setminus\{u\}$. This space can be decomposed as
$$E^*=\ell_1\oplus_1 M(\beta\en\setminus(\en\cup\{u\})).$$
The second dual is then represented as
$$E^{**}=\ell_\infty\oplus_\infty M(\beta\en\setminus(\en\cup\{u\}))^*.$$
Denote by $j$ the canonical embedding of $E$ into $E^{**}$ and by $\rho$ the embedding $\rho:\ell_\infty\to E^{**}$ given by $\rho(x)=(x,0)$ using the above representation. Now,
$$\rho(\ell_\infty)=\{x^{**}\in E^{**}: x^{**}(\mu)=0\mbox{ whenever $\mu\in M(\beta\en\setminus \{u\})$ is such that }\mu|_{\en}=0\}.$$
So, $\rho(\ell_\infty)$ is weak* closed and, moreover, $\rho$ is weak* to weak* homeomorphism ($\ell_\infty$ being considered as the dual to $\ell_1$).

Finally, $\rho|_{K_0}=(j\circ\kappa)|_{K_0}$ and hence $\wscl B=\rho\left(\overline{K_0}\right)$. Fix some $x\in\overline{K_0}\setminus K_0$ and let $A\subset\en$ be infinite such that $x|_A=1$. Such a set $A$ exists due to (c). Enumerate $A=\{a_n:n\in\en\}$ and define an element $u\in\ell_1$ by
$$u(k)=\begin{cases} \frac1{2^{n+1}}, & k = a_n,\\ 0, & k\in\en\setminus A.\end{cases}$$
Further define the element $x^{*}\in E^*$ by $x^*=(u,0)$ (using the above representation). Then $\|x^*\|=1$, so $\sup x^*(C)\le 1$. Moreover,
$\rho(x)(x^*)=1$, hence $\sup x^*(C)=1$. Let $x^{**}\in\wscl C$ be such that
$x^{**}(x^*)=1$. Then $x^{**}=(\rho(y),0)$ for some $y\in\ell_\infty$. As $\|y\|\le 1$, we get $y|_A=1$. But then $d(y,c_0)=1$, hence $d(x^{**},C)\ge 1$.
So, $\Ja(C)\ge 1$ and the proof is completed.
\end{proof}

The above examples show that any of the inequalities from Corollary~\ref{cor:AllPossibleInequalities} can be strict, with one possible exception which is described in the following problem.

\begin{question} Let $E$ be a Banach space and $C\subset E$ a bounded convex set. Is then $\Ja(C)=\ck(C)$?
\end{question}

\section{The case of weak* angelic dual unit ball}

In this section we collect several results saying that under some additional conditions some of the inequalities from Corollary~\ref{cor:AllPossibleInequalities} become equalities.
The basic assumption will be that the dual unit ball $B_{E^*}$ is weak* angelic, i.e. that whenever $A\subset B_{E^*}$ and $x^*\in\wscl A$, there is a sequence in $A$ which weak* converges to $x^*$. Inspired by \cite{FHMZ} we introduce the following quantity. If $E$ is a Banach space and $H\subset E$ a bounded subset, we set
$$\gamma_0(H)=\sup\{|\lim_i \lim_j x^*_i(x_j)| : (x_j)\subset H, (x^*_i)\subset B_{E^*}, x^*_i \overset{w^*}{\to} 0\},$$
assuming the involved limits exist. It is clear that $\gamma_0(H)\le \gamma(H)$.
In general $\gamma_0$ is not an equivalent quantity to the other ones. Indeed, if $E=\ell_\infty$ and $C=B_E$, then $\gamma_0(C)=0$ by the Grothendieck property of $E$. But in case $B_{E^*}$ is angelic, we have the following:

\begin{thm}\label{thm-angelic} Let $E$ be a Banach space such that $B_{E^*}$ is weak* angelic.
\begin{itemize}
\item Let $H\subset E$ be any bounded subset Then we have:
$$\frac12\gamma(H)\le\gamma_0(H)=\Jae(H)=\cke(H)=\dh(\wscl H,E)\le\gamma(H).$$

\item Let
$C\subset E$ be any bounded convex subset. Then the following inequalities hold true:
\begin{multline*}\frac12\gamma(C)\le\gamma_0(C)=\Jae(C)=\ck_E(C)=\dh(\wscl C,E) \\ \le \Ja(C)\le\ck(C)\le\dh(\wscl C,C)\le\gamma(C).\end{multline*}
\end{itemize}
\end{thm}

\begin{proof} The second part follows from the first part and  Corollary~\ref{cor:AllPossibleInequalities}.
As for the first part, in view of Corollary~\ref{cor:AllPossibleInequalities} it is enough to prove that $\Jae(H)\ge \gamma_0(H)$ and $\dh(\wscl H,E)\le\gamma_0(H)$. The second inequality follows from \cite[Proposition 14(ii)]{FHMZ}.

The first inequality follows from the proof of Theorem~\ref{prop-pryce}.
In fact, the angelicity assumption is not needed here.
Let us indicate the necessary changes:

Suppose that $\gamma_0(H)>r$. The space $F$ is not needed, but define the sublinear functional $p$ on $E^*$ by $p(f)=\sup f(H)$ for $f\in E^*$.
Fix a sequence $(z_j)$ in $H$ and $(f_i)$ in $B_{E^*}$ such that
$f_i$ weak* converge to $0$ and $\lim_i\lim_j f_i(z_j)>r$ and all the limits involved exist. Without loss of generality suppose that:
\begin{center}{\it 
for every  $i\in \en$,there is  $j_0\in\en$, such that for all $j\ge j_0$ we have $f_i(z_j)>r.$}
\end{center}
As $\limsup_i f_i=\liminf_i f_i=0$, we get the assertion (\ref{pecko}) for free.
We define $K_n$ for $n\in\en$ in the same way. The assertion (\ref{kjedna}) then says that $p(f)>r$ for all $f\in K_1$. Fix any $r'<r$ and a sequence $(\beta_n)$ of strictly positive numbers. Claim~\ref{cl3} now yield a sequence $(g_n)$ with $g_n\in K_n$ such that
$$ p\left(\sum_{i=1}^n \beta_i g_i\right)>\beta_n r'+p\left(\sum_{i=1}^{n-1} \beta_i g_i\right).$$
As $f_n$ weak* converge to $0$, $g_n$ weak* converge to $0$ as well. Thus $g_0=0$. Now, if the sequence $(\beta_n)$ quickly converges to $0$ (i.e., satisfies the same condition as in the original proof), we set
$g=\sum_{i=1}^\infty \beta_i g_i$.
Let $u\in\wscl H$ be an arbitrary point with $g(u)=\sup g(H)$.
By the final calculation we get $\liminf_n g_n(u)\ge r'$. If $v\in E$ is arbitrary, then $g_n(v)\to 0$, and thus
$$r'\le \liminf g_n(u)-\lim g_n(v)=\liminf g_n(u-v)\le \|u-v\|.$$
Thus $\Jae(H)\ge r'$, so $\Jae(H)\ge\gamma_0(H)$.
\end{proof}

Let us remark that the spaces from Examples~\ref{exa-c0},~\ref{exa-ell1} and~\ref{exa-c0omega} are separable and therefore they  have weak* angelic unit ball. It follows that in Theorem~\ref{thm-angelic} all the inequalities, with a possible exception of $\Ja(C)\le\ck(C)$, may be strict. {Note also that under the weaker assumption of the Banach space $E$ having Corson property $\cC$, it has been proved in~\cite[Proposition 2.6]{ang-cas1} that for any bounded set $H\subset E$ we have $\ck_E(H)=\dh(\wscl H,E)$.}

The following theorem shows that  all the quantities are equal in a very special case $E=c_0(\Gamma)$.

\begin{thm} Let $\Gamma$ be an arbitrary set and $E=c_0(\Gamma)$.
\begin{itemize}
	\item Let $H\subset E$ be a bounded set. Then we have:
	$$\gamma_0(H)=\Jae(H)=\cke(H)=\dh(\wscl H,E)=\gamma(H).$$

 \item Let $C\subset E$ be a convex bounded subset. Then we have:
$$\gamma_0(C)=\Jae(C)=\ck_E(C)=\dh(\wscl C,E) = \Ja(C)=\ck(C)=\dh(\wscl C,C)=\gamma(C).$$
\end{itemize}
\end{thm}

\begin{proof} It is enough to prove $\gamma_0(H)\ge \gamma(H)$. If $\gamma(H)=0$, this inequality is trivial. So, suppose that $\gamma(H)>0$.
Fix an arbitrary $r>0$ such that $\gamma(H)>0$. We find  sequences $(x_i)\subset H$, $(x^*_j)\subset B_{E^*}$ and $\eta>0$ such that
$$\lim_i \lim_j x_i^*(x_j)-\lim_j \lim_i x_i^* (x_j) > r(1+\eta),$$
where all the limits involved exist.
As $B_{E^*}$ is weak* sequentially compact, by passing to a subsequence
we may suppose that the sequence $(x_i^*)$ weak* converges to some $x^*\in B_{E^*}$. Then
$$\lim_i \lim_j (x_i^*-x^*)(x_j) > r(1+\eta).$$
We claim that
$$\limsup \|x_i^*-x^*\|\le 1.$$
Suppose not. Then, up to passing to a subsequence, we may suppose that there is $\delta>0$ such that $\|x_i^*-x^*\|\ge 1+\delta$ for each $i\in\en$. To proceed the proof we recall that $E^*$ is canonically identified with $\ell_1(\Gamma)$ an that the weak* topology on bounded sets coincides with the pointwise convergence topology. Using this identification we can find a finite set $F\subset \Gamma$ such that
$$\sum_{\gamma\in\Gamma\setminus F}|x^*(\gamma)|<\frac\delta3.$$
Further, as $x_i^*$ weak* converges to $x^*$, there is $i_0\in\en$ such that for each $i\ge i_0$ we have
$$\sum_{\gamma\in F}|x_i^*(\gamma)-x^*(\gamma)|<\frac\delta3.$$
Fix any $i\ge i_0$. Then we have:
\begin{align*}\|x_i^*\|&\ge \sum_{\gamma\in\Gamma\setminus F}|x_i^*(\gamma)|
\ge \sum_{\gamma\in\Gamma\setminus F}|x_i^*(\gamma)-x^*(\gamma)|
-\sum_{\gamma\in\Gamma\setminus F}|x^*(\gamma)|
\\ & =\|x_i^*-x^*\|-\sum_{\gamma\in F}|x_i^*(\gamma)-x^*(\gamma)|
-\sum_{\gamma\in\Gamma\setminus F}|x^*(\gamma)|
\\ 	& >1+\delta-\frac\delta3-\frac\delta3=1+\frac\delta3.
\end{align*}
This is a contradiction.

So, omitting finite number of elements, we can suppose that $\|x_i^*-x^*\|<1+\eta$ for all $i\in\en$. Set $y_i^*=\frac{x_i^*-x^*}{1+\eta}$.
Then $y_i^*\in B_{E^*}$, the sequence $(y_i^*)$ weak* converges to $0$ and
$$\lim_i \lim_j y_i^*(x_i)>r.$$
Thus $\gamma_0(H)\ge r$ and the proof is completed.
\end{proof}

The equalities $\ck_E(H)=\dh(\wscl H,E) =\gamma(H)$ in case $E=c_0$ and $H\subset E$ is a bounded subset follow also easily from \cite[Theorem 2.8]{kr-pr-sc}, see also \cite[Corollary 3.4.3]{ang-tesis}.


\end{document}